\documentclass[12pt]{amsart}
\usepackage{amscd,amsmath,amsthm,amssymb}
\usepackage[left]{lineno}
\usepackage{tikz}
\usepackage{color}
\usepackage{stmaryrd}
\usepackage[utf8]{inputenc}
\definecolor{verylight}{gray}{0.97}
\definecolor{light}{gray}{0.9}
\definecolor{medium}{gray}{0.85}
\definecolor{dark}{gray}{0.6}

 \def\NZQ{\mathbb}               
 \def\NN{{\NZQ N}}
 \def\QQ{{\NZQ Q}}
 \def\ZZ{{\NZQ Z}}

 \def\FF{{\NZQ F}}

 \def\frk{\mathfrak}               

 \def\mm{{\frk m}}

 \def\G{{\mathcal G}}

 \def\P{{\mathcal P}}

 \def\xb{{\mathbf x}}

 \def\opn#1#2{\def#1{\operatorname{#2}}} 
 \opn\chara{char} \opn\length{\ell} \opn\pd{pd} \opn\rk{rk}
 \opn\projdim{proj\,dim} \opn\injdim{inj\,dim} \opn\rank{rank}
 \opn\depth{depth} \opn\grade{grade} \opn\height{height}
 \opn\embdim{emb\,dim} \opn\codim{codim}
 
 \opn\Tr{Tr} \opn\bigrank{big\,rank}
 \opn\superheight{superheight}\opn\lcm{lcm}
 \opn\trdeg{tr\,deg}
 \opn\reg{reg} \opn\lreg{lreg} \opn\ini{in} \opn\lpd{lpd}
 \opn\size{size} \opn\sdepth{sdepth}
 \opn\link{link}\opn\fdepth{fdepth}\opn\lex{lex}
 \opn\tr{tr}
 \opn\type{type}
 \opn\gap{gap}
 \opn\arithdeg{arith-deg}
 \opn\astab{astab}
  \opn\dstab{dstab}
  \opn\pol{pol}
  \opn\mat{mat}
  \opn\indmat{indmat}
  \opn\sat{sat}
 \opn\div{div} \opn\Div{Div} \opn\cl{cl} \opn\Cl{Cl}
 \opn\Spec{Spec} \opn\Supp{Supp} \opn\supp{supp} \opn\Sing{Sing}
 \opn\Ass{Ass} \opn\Min{Min}\opn\Mon{Mon}
 \opn\Ann{Ann} \opn\Rad{Rad} \opn\Soc{Soc}
 \opn\Im{Im} \opn\Ker{Ker} \opn\Coker{Coker} \opn\Am{Am}
 \opn\Hom{Hom} \opn\Tor{Tor} \opn\Ext{Ext} \opn\End{End}
 \opn\Aut{Aut} \opn\id{id}
 
 \opn\nat{nat}
 \opn\pff{pf}
 \opn\Pf{Pf} \opn\GL{GL} \opn\SL{SL} \opn\mod{mod} \opn\ord{ord}
 \opn\Gin{Gin} \opn\Hilb{Hilb}\opn\sort{sort}
 \opn\PF{PF}\opn\Ap{Ap}
 \opn\mult{mult}
 \opn\aff{aff}
 \opn\relint{relint} \opn\st{st}
 \opn\lk{lk} \opn\cn{cn} \opn\core{core} \opn\vol{vol}  \opn\inp{inp} \opn\nilpot{nilpot}
 \opn\link{link} \opn\star{star}\opn\lex{lex}\opn\set{set}
 \opn\width{wd}
 \opn\Fr{F}
 \opn\QF{QF}
 \opn\G{G}
 \opn\type{type}\opn\res{res}
 \opn\conv{conv}
 \opn\gr{gr}
 
 \def\pot#1#2{#1[\kern-0.28ex[#2]\kern-0.28ex]}

 \opn\dirlim{\underrightarrow{\lim}}
 \opn\inivlim{\underleftarrow{\lim}}
 \let\union=\cup
 \let\sect=\cap

 \let\iso=\cong
 \let\Union=\bigcup
 \let\Sect=\bigcap
 \let\Dirsum=\bigoplus
 
 \let\to=\rightarrow
 
 \def\Implies{\ifmmode\Longrightarrow \else
         \unskip${}\Longrightarrow{}$\ignorespaces\fi}
 \def\implies{\ifmmode\Rightarrow \else
         \unskip${}\Rightarrow{}$\ignorespaces\fi}
 \def\iff{\ifmmode\Longleftrightarrow \else
         \unskip${}\Longleftrightarrow{}$\ignorespaces\fi}

 \let\:=\colon
 \newtheorem{Theorem}{Theorem}[section]
 \newtheorem{Lemma}[Theorem]{Lemma}
 \newtheorem{Corollary}[Theorem]{Corollary}
 \newtheorem{Proposition}[Theorem]{Proposition}
 \newtheorem{Remark}[Theorem]{Remark}
 
 \newtheorem{Example}[Theorem]{Example}
 
 \newtheorem{Definition}[Theorem]{Definition}

 \let\epsilon\varepsilon
 \let\kappa=\varkappa
 \opn\dis{dis}
 \def\pnt{{\raise0.5mm\hbox{\large\bf.}}}
 
 \opn\Lex{Lex}

\begin{document}

\title {The saturation number of powers of graded ideals}

\author {J\"urgen Herzog, Shokoufe Karimi and Amir Mafi}

\address{J\"urgen Herzog, Fachbereich Mathematik, Universit\"at Duisburg-Essen, Campus Essen, 45117
Essen, Germany} \email{juergen.herzog@uni-essen.de}

\address{Sh. Karimi, Department of Mathematics, University of Kurdistan, P.O. Box: 416, Sanandaj,
Iran.}
\email{shkarimi2011@yahoo.com }

\address{A. Mafi, Department of Mathematics, University of Kurdistan, P.O. Box: 416, Sanandaj,
Iran.}
\email{a\_mafi@ipm.ir}

\dedicatory{Dedicated to the memory of J\"urgen Herzog}

\begin{abstract}
Let $S=K[x_1,\ldots,x_n]$ be the polynomial ring in $n$ variables over a field $K$ with maximal ideal $\mm=(x_1,\ldots,x_n)$,  and let $I$ be a graded ideal of $S$. In this paper, we define the  saturation number $\sat(I)$ of $I$ to be  the smallest non-negative integer $k$ such that $(I:\mm^{k})= (I:\mm^{k+1})$.  We show that $\sat(I^k)$ is linearly bounded. Furthermore, we show that $\sat(I^k)=k$ if $I$ is a principal Borel ideal and prove that $\sat(I_{d,n}^k) =\max\{l\:\; (kd-l)/(k-l) \leq n, \   l\leq k\},$ where  $I_{d,n}$ is the squarefree Veronese ideal generated in degree $d$. A general strategy  is given to compute  $\sat(I)$ when $I$ is a polymatroidal ideal.
\end{abstract}

\thanks{This paper was done while the first author
was visiting University of Kurdistan, Iran.}




\maketitle

\section*{Introduction}

Throughout this paper, we assume that $S=K[x_1,\ldots,x_n]$ is the polynomial ring in $n$ variables over a field $K$ with the unique graded maximal ideal $\mm=(x_1,\ldots,x_n)$, and that $I$ is a graded ideal of $S$. If $I$ is a monomial ideal, then we denote by $G(I)$ the unique minimal set of monomial generators of $I$.

The ideal
\[
I^{\sat}:= (I:\mm^\infty) =\Union_{k\geq 0}(I:\mm^k)
\]
is called the {\em saturation} of $I$. Since $I\subseteq (I:\mm)\subseteq (I:\mm^2)\subseteq \ldots$ and $S$ is Noetherian, there exists an integer $k$ such that $(I:\mm^{k})=( I:\mm^{k+1})$.  This implies that $\lambda(I^{\sat}/I)<\infty$. Here $\lambda(M)$ denotes the length of a module $M$.

In the first part of this paper we introduce  $\sat(I)$  of $I$ which is defined to be the smallest non-negative integer $k$ such that $(I:\mm^{k})=(I:\mm^{k+1})$. We show that $\sat(I)\leq \reg(I)$ and obtain from this that $\sat(I^k)$ is bounded by a linear function of $k$. It would be interesting to have this linear bound also for regular local rings. For monomial ideals one obtains an even better result. Indeed, we show that there exists a quasi-linear function $f$ such that $f(k)=\sat(I^k)$ for $k\gg 0$ if all powers of $I$ have linear resolution. This is a consequence of the fact that $A=\Dirsum_{k\geq 0} (I^k)^{\sat}$ is finitely generated when $I$ is a monomial ideal.

In the second part we study the saturation number of special classes of polymatroidal ideals.
A polymatroidal ideal is a monomial ideal, generated in a single degree $d$ satisfying the following condition: for all monomials $u,v\in G(I)$ with $\deg_{x_i}(u)>\deg_{x_i}(v)$, there exists an index $j$ such that $\deg_{x_j}(v)>\deg_{x_j}(u)$ and $x_j(u/x_i)\in I$ (see \cite{HH} or \cite{HH2}), in particular a squarefree polymatroidal ideal is called a matroidal ideal. Among the stable ideals, the principal Borel ideals are polymatroidal. Note that a monomial ideal $I$ is called stable if for all monomials $u\in I$ and all $i<m(u)$ one has $x_iv\in I$, where $v=u/x_{m(u)}$ and $m(u)$ denotes the largest index $j$ such that $x_j$ divides $u$.
The saturation number for this class of ideals behaves particularly nice. We show that $\sat(I^k)=k$ for all $k$. The squarefree principal ideals are matroidal. Here the situation is more complicated. In this paper only consider the squarefree Veronese ideals, that is, the squarefree monomial ideal generated by all monomials of the same degree $d$. We denote by $I_{d,n}$ (see \cite{HH} or \cite{HRV}) the squarefree Veronese ideal of degree $d$ in the variables $x_1,\ldots,x_n$, and show that $\sat(I_{d,n}^k) =\max\{l\:\; (kd-l)/(k-l) \leq n,\ l\leq k\}$ for all $k$. According to \cite[Proposition 5]{HV} any polymatroidal ideal is of intersection type. This fact, allows us to compute the saturation number of a polymatroidal ideal in each concrete case. Moreover, it can be shown that if $I$ is any monomial ideal of intersection type with $\Ass(I)=\Ass^\infty(I)$, then $\sat(I^k)=\sat(I)k$. Here $\Ass^\infty(I)= \Union_{k\geq 0}\Ass(I^k)$ .

For unexplained notation or terminology, we refer the reader to \cite{HH2} and \cite{BH}. Several explicit examples were  performed with help of the computer algebra systems Macaulay2 \cite{GS} and CoCoA \cite{AB}.

\section{Upper bounds for the saturation number of an ideal and its powers}
We start this section by the following definition.

\begin{Definition}{\em
The number $\sat(I)=\min\{k\:\; (I:\mm^{k})= (I:\mm^{k+1})\}$ is called the {\em saturation  number} of $I$. }
\end{Definition}

It is clear that $\sat(I)=0$ if and only if $\depth(S/I)>0$. In particular, $\sat(I)=0$ if $I$ is a squarefree  monomial ideal.

Let $M$ be a finitely generated graded $S$-module.  We set
\[
\alpha(M)=\min\{k\:\; M_k\neq 0\}\quad\text{and}\quad \beta(M)=\max\{k\:\; M_k\neq 0\},
\]
if minimum and maximum exist and otherwise we set $-\infty$ and $\infty$ respectively.
Note that if $M$ is a finite length graded $S$-module, then $\beta(M)$ is the largest degree of a socle element of $M$. For a nonzero finite length graded $S$-module $M$, we set $\sigma(M)=\beta(M)-\alpha(M)+1$ and when $M=0$ we assume $\sigma(M)=0$  and also we consider $\gamma(M)=\min\{k\:\; (0:_M \mm^k)=M\}$.

\begin{Lemma}
\label{sigma}
Let $M\neq 0$ be a graded $S$-module of finite length.  Then
\[
\gamma(M)\leq \sigma(M).
\]
\end{Lemma}

\begin{proof}
We proceed by induction on $\gamma(M)$. The assertion is trivial if  $\gamma(M)=1$. Let $\gamma(M)>1$.  Then  $\gamma(M/(0:_M\mm))=\gamma(M)-1$, $\beta(M/(0:_M\mm))\leq \beta(M)-1$ and $\alpha( M/(0:_M\mm))\geq \alpha(M)$. Thus, together with the induction hypothesis we obtain
\[
\gamma(M)-1=\gamma(M/(0:_M\mm))\leq \sigma(M/(0:_M\mm))\leq\sigma(M)-1,
\]
as desired.
\end{proof}

Applied to the saturation of an ideal we obtain

\begin{Corollary}
\label{better}
Let $I\subset S$ be a graded ideal. Then
\[
\sat(I)\leq \sigma(I^{\sat}/I).
\]
\end{Corollary}

In general the inequality $\sat(I)\leq \sigma(I^{\sat}/I)$  may be strict. The following example was communicated to us by Dancheng Lu and Lizhong Chu: \\
let $I=(xy,yz,zu,uv, vx^2)^5$.  Then $\sat(I)=2$ and $\gamma(I^{\sat}/I)\geq 4$.

\medskip
The following result describes a case where $\sat(I)=\sigma(I^{\sat}/I)$.

\begin{Proposition}
\label{yes}
Let $I$ be a graded ideal with $d$-linear resolution and such that $\depth S/I=0$. Then $\sat(I)=\sigma(I^{\sat}/I)$.
\end{Proposition}

\begin{proof}
Since $\depth S/I=0$ it follows that $\projdim S/I=n$, Let $\FF$ be the graded minimal free resolution of $S/I$. Since $F_n/\mm F_n \iso H_n(\xb;S/I)$, where $H_n(\xb;S/I)$ denotes the Koszul homology with respect to $\xb=x_1,\ldots,x_n$,  we see that the degree of the socle elements of $S/I$ correspond to the degree $-n$ of the generators of the graded free module $F_n$. Since $I$ has a $d$-linear resolution it follows that all generators of $F$ are of degree $d+n-1$. Therefore all socle elements of $S/I$ are of degree $d-1$. It follows that $I:\mm$ is generated in degree $d-1$ and $d$.

In view of  Corollary~\ref{better} it suffices to show that $\sat(I)\geq \sigma(I^{\sat}/I)$. Let $a=\alpha(I^{\sat})$ and $s=\sat(I)$. Then
$I^{\sat}=I:\mm^s=(I:\mm):\mm^{s-1}$. Therefore, $\mm^{s-1} I^{\sat} \subset I:\mm$. This implies that $a+(s-1)\geq d-1$, and hence $s\geq (d-1)-a+1 =\sigma(I^{\sat}/I)$.
\end{proof}

\medskip
We denote by $\reg(M)$ the Castelnuovo-Mumford regularity of a finitely generated graded $S$-module.

\begin{Corollary}
\label{reg}
Let $I\subset S$ be a graded ideal. Then
$
\sat(I)\leq \reg(I).
$
\end{Corollary}

\begin{proof}
It follows from  Corollary~\ref{better} that $\sat(I)\leq \beta(I^{\sat}/I)+1$. Since $I^{\sat}/I$ is of finite length, we have $\beta(I^{\sat}/I)=\reg(I^{\sat}/I)$. By \cite[Lemma 1.1]{CH} or \cite[Corollary 20.19]{E}, $\reg(I^{\sat}/I)\leq \reg(S/I)=\reg(I)-1$. The desired conclusion follows.
\end{proof}

The following result follows by Corollary~\ref{reg} and~\cite{Ko} (see also ~\cite{CHT}).
\begin{Corollary}
\label{linearbound}
Let $I\subset S$ be a graded ideal. Then there exists a linear function $f$ such that $\sat(I^k)\leq f(k)$ for all $k\geq 0$.
\end{Corollary}

Observe that $I=I^{\sat}\sect Q$ where $Q$ is an $\mm$-primary ideal. We call $I$ {\em special} if $Q$ is a power of $\mm$. In particular, $I^{\sat}\sect \mm^d$ is special for all $d$. In the sequel we will use

\begin{Proposition}
\label{chinesecopy}
Let $I\subset S$ a graded ideal with $I=I^{\sat}$. Then
 \[
\sat(I\cap \mm^d) = \left\{
\begin{array}{ll}
0,  &   \text{if $d\leq \alpha(I)$,}\\
d-\alpha(I), & \text{if $d>\alpha(I)$.}\\
\end{array}
\right. \]
\end{Proposition}

\begin{proof}
If $d\leq \alpha(I)$, then $I\cap \mm^d=I$, and since $I$ is saturated we have $\sat(I\cap \mm^d)=0$. Now, let $d> \alpha(I)$. Since $\mm^d:\mm=\mm^{d-1}$ and $(I:\mm)=I$ because of  $I$ is saturated, it easily follows that  $(I\sect \mm^d):\mm=I\sect \mm^{d-1}$. By induction hypothesis we obtain
\[
\sat(I\sect \mm^d)=\sat((I\sect \mm^d):\mm)+1=\sat(I\sect \mm^{d-1})+1=(d-1)-\alpha(I) +1 =d-\alpha(I),
\]
as desired.
\end{proof}

For  monomial ideals, Corollary~\ref{linearbound} can be improved as follows: a function $f\:\; \ZZ\to  \ZZ$ is called {\em quasi-linear}, if
there exists an integer $d\geq 1$ and for $i=0,\ldots,d-1$, linear functions $f_i(x) =a_ix+b_i$ with $a_i,b_i\in\QQ$ such that $f(k)=f_i(k)$ for $k\equiv i\mod d$.

\begin{Theorem}
\label{quasilinear}
 Let $I\subset S$ be a monomial ideal. Then there exists a quasi-linear function $f$ such that $\sigma((I^k)^{\sat}/I^k)=f(k)$ for $k\gg 0$.
\end{Theorem}

\begin{proof}
Put $M_k=(I^k)^{\sat}/I^k=H^0_{\mm}(S/I^k)$. This is a $\ZZ^n$-graded $S$-module of finite length, and since $I^k$ and $(I^k)^{\sat}$ are monomial ideals, the monomials in $(I^k)^{\sat}\setminus I^k$ form a $K$-basis of $M_k$. If $M_k\neq 0$, then $f(k)=\beta(M_k)-\alpha(M_k)+1$.

\smallskip\noindent\emph{Step 1: Reduction.} By Brodmann \cite{B}, $\depth S/I^k$ is constant for all $k\gg 0$. If this constant is positive, then $I^k$ is saturated, $M_k=0$ and $f(k)=0$ for $k\gg0$. So we may assume that there is $k_0$ with $\depth S/I^k=0$ for all $k\geq k_0$. Then $\mm\in\Ass(S/I^k)$, hence $M_k\neq0$ for all $k\geq k_0$.

\smallskip\noindent\emph{Step 2: $\beta(M_k)$ is linear for $k\gg0$.} The top nonzero graded component of a finite length module is annihilated by $\mm$, so $\beta(M_k)$ is the largest degree of a socle element of $M_k$. Since $\depth S/I^k=0$, we have $(0:_{S/I^k}\mm)\subseteq H^0_{\mm}(S/I^k)=M_k$, and therefore $\operatorname{Soc}(M_k)=\operatorname{Soc}(S/I^k)$. Because $H_n(\xb;S/I^k)\iso \operatorname{Soc}(S/I^k)(-n)$, we obtain
\[
\beta(M_k)=\max\{j\:\; \beta_{n,j}(S/I^k)\neq 0\}-n .
\]
By \cite[Theorem 3.1]{CHT} (see also \cite{Ko}), the right hand side is a linear function of $k$ for $k\gg0$.

\smallskip\noindent\emph{Step 3: $\alpha(M_k)$ is quasi-linear for $k\gg0$.} Note that $\alpha(M_k)$ is \emph{not} in general the least degree of a generator of $(I^k)^{\sat}$ (see Remark~\ref{rem:alpha} below), so we argue differently. Let
\[
\mathcal{R}=\Dirsum_{k\geq0}I^kt^k\quad\text{and}\quad A=\Dirsum_{k\geq0}(I^k)^{\sat}t^k .
\]
Both are $\ZZ^{n+1}$-graded subalgebras of $S[t]$ spanned by monomials. $\mathcal R$ is finitely generated, and since $I$ is a monomial ideal, $A$ is finitely generated by \cite[Theorem 3.2]{HHT}. Hence the sets
\[
\Sigma_{A}=\{(a,k)\in\NN^n\times\NN\:\; x^a\in (I^k)^{\sat}\},\qquad
\Sigma_{\mathcal R}=\{(a,k)\in\NN^n\times\NN\:\; x^a\in I^k\}
\]
are finitely generated submonoids of $\NN^{n+1}$. Finitely generated monoids are images of $\NN^r$ under linear maps, hence \emph{semilinear} (finite unions of sets $b+\NN p_1+\cdots+\NN p_s$). By Ginsburg and Spanier \cite{GSp}, semilinear sets are exactly the sets definable by Presburger formulas; in particular they are closed under finite unions, intersections, complements, and linear images. Consider
\[
P=\{(d,k)\in\NN^2\:\; \text{there is } (a,k)\in\Sigma_A\setminus\Sigma_{\mathcal R}\text{ with } a_1+\cdots+a_n=d\},
\]
and
\[
\Gamma=\{(d,k)\in P\:\; k\geq k_0,\ (d',k)\notin P\text{ for all }d'<d\}.
\]
Both are semilinear, being defined by Presburger formulas. By the description of the basis of $M_k$ above, $\Gamma$ is exactly the graph of the function $k\mapsto\alpha(M_k)$, $k\geq k_0$. So it suffices to prove:

\smallskip\noindent\emph{Claim.} Let $g\:\; \NN_{\geq k_0}\to\NN$ be a function whose graph $\Gamma\subset\NN^2$ is semilinear. Then $g$ is quasi-linear for $k\gg0$.

\smallskip\noindent\emph{Proof of the Claim.} Write $\Gamma=L_1\cup\cdots\cup L_m$ with $L_j=b_j+\NN p_{j1}+\cdots+\NN p_{js_j}$, where we may assume that all $p_{ji}\neq 0$. Since $\Gamma$ is a graph, no period $p_{ji}=(p',p'')$ can have $p'=0$ (otherwise $(k,y)$ and $(k,y+p'')$ would both lie in $\Gamma$). If two periods $p=(p',p'')$ and $q=(q',q'')$ of the same $L_j$ satisfied $p''/p'\neq q''/q'$, then $b_j+q'p$ and $b_j+p'q$ would be two points of $\Gamma$ with the same first coordinate and different second coordinates. So all periods of $L_j$ lie on a line through the origin of some rational slope $c_j$, and $L_j\subset\{(k,c_jk+e_j)\}$ for some $e_j\in\QQ$. If $s_j\geq1$, let $g_j$ be the gcd of the first coordinates of the periods of $L_j$. Every sufficiently large integer $k\equiv b_j'\mod g_j$ is the first coordinate of a point of $L_j$ (every large multiple of $g_j$ is a nonnegative combination of the first coordinates of the periods), so $g(k)=c_jk+e_j$ for all large $k\equiv b_j'\mod g_j$.

Let $q$ be the least common multiple of the $g_j$ with $s_j\geq 1$ and fix a residue $r$ modulo $q$. Only finitely many $k$ lie in sets $L_j$ with $s_j=0$, and every $k\ge k_0$ lies in some $L_j$. Hence for the residue class $r$ there are $k\equiv r$ (arbitrarily large) belonging to some $L_j$ with $s_j\geq1$, and then $k\equiv b_j'\mod g_j$ holds for all $k\equiv r\mod q$. Consequently $g(k)=c_jk+e_j$ for all large $k\equiv r\mod q$. This shows that $g$ is quasi-linear for $k\gg0$. \hfill$\square$

\smallskip\noindent\emph{Step 4: Conclusion.} For $k\gg0$ we have $f(k)=\beta(M_k)-\alpha(M_k)+1$, where $\beta(M_k)$ is linear by Step 2 and $\alpha(M_k)$ is quasi-linear by Step 3. Hence $f$ is quasi-linear for $k\gg0$.
\end{proof}

\begin{Remark}\label{rem:alpha}
In general $\alpha(M)$ for $M=I^{\sat}/I$ is not the least degree of a generator of $I^{\sat}$. For example, let $I=(x,y^3,y^2z^2)=(x,y^2)\cap(x,y^3,z^2)\subset K[x,y,z]$. Then $I^{\sat}=(x,y^2)$ has a generator of degree $1$, but the least degree of a nonzero element of $I^{\sat}/I$ is $2$ (it is the class of $y^2$). The two numbers do agree, for instance, when $I$ has a linear resolution; this is the situation of Proposition~\ref{yes}.
\end{Remark}

\section{The saturation number for polymatroidal ideals}

As a first example  of polymatroidal ideals we consider (squarefree) principal Borel ideals.
Let $u=x_1^{a_1}\cdots x_n^{a_n}$ be a monomial in $S=K[x_1,\ldots,x_n]$. We set $\deg_{x_i}(u)=a_i$ for $i=1,\ldots,n$. For a given integer  $k\geq 1$,  we let  $I^{\leq k}$ be the ideal generated by all $u\in G(I)$ with $\deg_{x_i}(u)\leq k$ for $i=1,\ldots,n$. The  ideal $I$ is squarefree if and only if  $I=I^{\leq 1}$.

 \begin{Definition}
\label{monloc}{\em Let $I\subset S=K[x_1,\ldots,x_n]$ be a monomial ideal, and let $k\geq 1$ be an integer. Then $I$ is called {\em $k$-strongly stable}, if
\begin{enumerate}
\item[(i)] $I=I^{\leq k}$;
\item[(ii)] for all  $u\in G(I)$ and all integers  $1\leq i<j\leq n$ with   $\deg_{x_j}(u)>0$ and  $\deg_{x_i}(u)\leq k-1$ it follows that $x_i(u/x_j)\in I$.
\end{enumerate}
}
\end{Definition}

The ideal $I$ is called {\em squarefree strongly stable}, if it is $1$-strongly stable. $I$ is {\em strongly stable}, if $I$ is $k$-strongly stable for $k$ bigger than the maximal degree of  a monomial in $G(I)$. In other words, if $u\in G(I)$ and $x_j$ divides $u$, then $x_i(u/x_j)\in I$ for all $i\leq j$.

Let $u_1,\ldots,u_m$ be monomials in $S$ with $\deg_{x_i}(u_j)\leq k$ for $i=1,\ldots,n$ and $j=1,\ldots,m$.  There exists a unique smallest $k$-strongly stable ideal containing $u_1,\ldots,u_m$  which we denote by $B^k(u_1,\ldots,u_m)$. The monomials $u_1,\ldots u_m$ are called  {\em Borel generators} of $B^k(u_1,\ldots,u_m)$. A monomial ideal $I$ is called {\em $k$-principal Borel} if $I=B^k(u)$ for some monomial $u$ with $\deg_{x_i}(u)\leq k$ for $i=1,\ldots,n$, in particule $1$-principal Borel ideals are squarefree principal Borel ideals, see \cite[Theorem 3.1]{HLR}.
For simplify notation, we write $B(u_1,\ldots,u_m)$ instead of $B^1(u_1,\ldots,u_m)$ where $u_1,\ldots,u_m$ are squarefree monomials.

Let $u, v$ be monomials of same degree and assume  that $\deg_{x_i}(u)\leq k$ for $i=1,\ldots,n$. Then we write $v\preceq_k u$ if $v\in B^k(u)$, and $v\prec_k u$ if $v\in B^k(u)$ and $v\neq u$. For each $d\geq 1$, $v\preceq_k u$ defines a partial order on the set of monomials of degree $d$ whose exponents are bounded by $k$.

\begin{Theorem}
\label{nicetheorem}
Let $u\in S$ be a monomial  with $\deg_{x_n}(u)>0$, and let $I=B(u)$. Then $\sat(I^k)=k$.
\end{Theorem}

\begin{proof}
It is observed in \cite{HH} that principal Borel ideals  are polymatroidal. Therefore these ideals  are of intersection type, as shown in \cite[Proposition 5]{HV}. Let $u=x_{i_1}x_{i_2}\cdots x_{i_d}$ with $1\leq i_1\leq i_2\leq \cdots \leq i_d=n$. Then $I=\prod_{j=1}^d(x_1,\ldots,x_{i_j})$, (see for example \cite[Examples 2.8]{HPV}). \\
We claim that  $\Ass(I)=\{(x_1,\ldots,x_{i_1}),(x_1,\ldots,x_{i_2}),\ldots, (x_1,\ldots,x_{i_d})\}$.
We prove by induction on $d$. If $d=1$, then there is nothing to prove. Let $d\geq 2$ and the claim has been proved for fewer than $d$.
It is clear that $\Ass(I)\setminus\{\mm\}=\cup_{j=2}^d\Ass(I:x_{i_j}^{\infty})$. Since $(I:x_{i_j}^{\infty})=\prod_{t=1}^{j-1}(x_1,\ldots,x_{i_t})$, we can apply the inductive hypothesis to deduce  our claim.
Therefore, by applying \cite[Corollary 4.10]{HRV}, we see that
\[
I^k=\Sect_{j=1}^{d-1}(x_1,\ldots,x_{i_j})^{kj}\sect \mm^{kd}.
\]
A monomial of least degree in
$\Sect_{j=1}^{d-1}(x_1,\ldots,x_{i_j})^{kj}$ is $x_1^{k(d-1)}.$
Finally we apply Proposition~\ref{chinesecopy} we get that $\sat(I^k)=kd-k(d-1)=k$, as desired.
\end{proof}

\begin{Lemma}
\label{bound}
Let $I$ be a polymatroidal ideal  and $k\geq 1$ an integer. Then $I^{\leq k}$ is a polymatroidal ideal. In particular, if $u$ is $k$-bounded monomial. Then $B^k(u)$ is a polymatroidal ideal.
\end{Lemma}

\begin{proof}
We show that for $u,v\in G(I^{\leq k})$ the exchange property holds. Indeed, let $i$ be such that $\deg_{x_i}(u)>\deg_{x_i}(v)$, since $I$ is polymatroidal, there exists $j$ such that $\deg_{x_j}(v)>\deg_{x_j}(u)$ and $x_j(u/x_i)\in I$. Since $v\in G(I^{\leq k})$, it follows that $\deg_{x_j}(v)\leq k$. Therefore, $\deg_{x_j}(u)\leq k-1$. This implies that $x_j(u/x_i)\in G(I^{\leq k})$.
\end{proof}

\begin{Lemma}
\label{hot}
Let $u\in S$ be a $k$-bounded monomial of degree $d$. Then $B^k(u):\mm  =I+B^k(u)$,  where  $I=0$  or $I$  is a $(k-1)$-Borel ideal generated in degree $d-1$. \end{Lemma}

\begin{proof}
 We may assume that $\depth S/B^k(u)=0$. By Lemma~\ref{bound},  $B^k(u)$  is a polymatroidal ideal, and hence has a $d$-resolution, see \cite[Lemma 1.3 ]{HT}. By \cite[Theorem 1.5]{AHZ}, it follows that $B^k(u):\mm  =I+B^k(u)$, where $I$ is generated in degree $d-1$.  Let $v\in G(I)$. Then $x_iv\in G(B^k(u))$ for all $i$.   Suppose $\deg_{x_i}(v)\geq k$ for some $i$. Then $\deg_{x_i}(x_iv)>k$, a contradiction. This shows that all generators of $I$ are $(k-1)$-bounded.

 To complete the proof we must show that $I$ is $(k-1)$-Borel. Let $v\in G(I)$ with $x_j|v$,  $i<j$ and   $\deg_{x_i}(x_iv)\leq k-1$.  Then we must show that $v_0=x_i(v/x_j)\in G(I)$, that is, $x_lv_0\in B^k(u)$ for $l=1,\ldots,n$.  Indeed, if $l=j$,  then  $x_{l}v_{0}=x_{i}v\in B^k(u)$. If  $l\neq j$, then $x_lv_0= x_i(x_lv)/x_j\in B^k(u)$ because  $x_lv\in B^k(u)$ and $B^k(u)$ is $k$-Borel.
\end{proof}

The next result  is taken from \cite{HLR}.

\begin{Lemma}
\label{rahmaticopy}
Let $u, v$ be $k$-bounded monomials of degree $d$. Let $u=x_{j_1}\ldots x_{j_d}$  with $j_1\leq j_2\leq \ldots\leq j_d$  and  $v=x_{i_1}\ldots x_{i_d}$ with $i_1\leq i_2\leq \ldots \leq i_d$.  Then $v\prec_k u$ if and only if $i_r\leq j_r$  for $r=1,\ldots, d$.
\end{Lemma}

\begin{Theorem}
\label{borelsocle}
Let $S=K[x_1,\ldots,x_n]$ and $k$, $d$ be positive integers.  Let $d=qk+r $ with integers $q,r\geq 0$  and $r<k$. We set  $u_{k,d,n}=x_n^kx_{n-1}^k\ldots x_{n-q+1}^kx_{n-q}^r$. If $r=0$, then  $u_{k,d,n}$ is defined in $S$ if and only if  $\lfloor d/k\rfloor\leq n$ and if $r\neq 0$ then $u_{k,d,n}$ is defined in $S$ if and only if $\lfloor d/k\rfloor< n$, and we have
\[
B^k(u_{k,d,n}):\mm =B^{k-1}(u_{k-1,d-1,n})+B^k(u_{k,d,n}),
\]
if $u_{k-1, d-1,n}$ is defined in $S$. Otherwise, $B^k(u_{k,d,n}):\mm =B^k(u_{k,d,n})$.
\end{Theorem}

\begin{proof}
We may assume that $k< d$. Indeed, if $k\geq  d$, then $u_{k,d,n}=x_n^d$ and $u_{k-1,d-1,n}=x_n^{d-1}$. In this case the assertion is obvious.

We first show that  $u_{k-1,d-1,n}\in B^k(u_{k,d,n}):\mm$. For this it suffices to show that $x_nu_{k-1,d-1,n}\in B^k(u_{k,d,n})$, because $B^k(u_{k,d,n})$ is $k$-Borel. Indeed, let $u_{k,d,n}=x_{i_1}x_{i_2} \cdots x_{i_d}$ with $i_1\leq i_2\leq \ldots \leq i_d=n$ and $x_nu_{k-1,d-1,n}=x_{j_1}x_{j_2} \cdots x_{j_d}$ with $j_1\leq j_2\leq \ldots \leq j_d=n$. By Lemma~\ref{rahmaticopy} we must show that $j_l\leq i_l$ for $l=1,\ldots, d$. Figure~\ref{comparison} illustrates this comparison. The integers $i_l$ and $j_l$ are labeled from right to left. Then for any $l$ with $1\leq l\leq d$ the boxes with the same $x$ coordinate give us the value of $i_l$ and $j_l$. For example, for $l_1$ in Figure~\ref{comparison} we obtain $i_{l_1}=n-q+1$ and $j_{l_1}=n-q'+1$, and for $l_2$  we obtain $i_{l_2}=n-2$ and $j_{l_2}=n-3$. From the equations $d=qk+r$ and $d-1=q'(k-1)+r'$ with $0\leq r<k$ and $0\leq r'<k-1$, it follows that $q\leq q'$ because $d/k\leq (d-1)/(k-1)$. Therefore Figure~\ref{comparison} shows that the desired inequalities  hold.
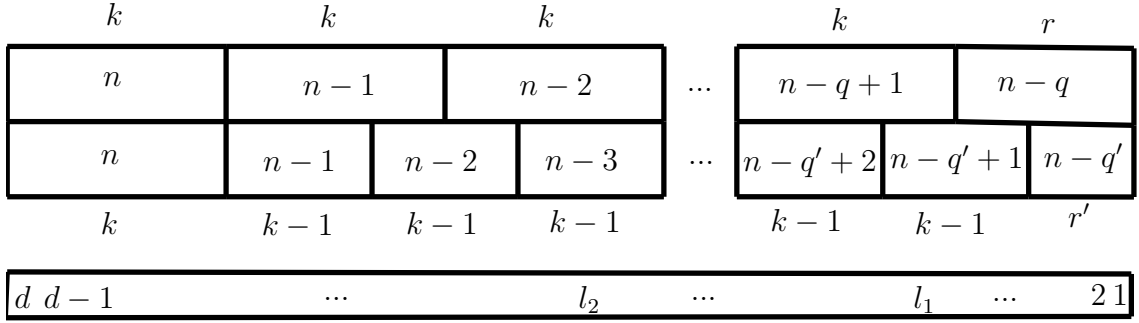
\begin{figure}[hbt]
\begin{center}
\begin{tikzpicture}[x=0.88cm,y=0.9cm,line width=1pt]
\draw (-6.,6.)--(-6.,5.);
\draw (-6.,5.)--(-3.,5.);
\draw (-3.,5.)--(-3.,6.);
\draw (-3.,6.)--(-6.,6.);
\draw (-3.,6.)--(0.,6.);
\draw (0.,6.)--(0.,5.);
\draw (0.,5.)--(-3.,5.);
\draw (-6.,6.)--(-9.,6.);
\draw (-9.,6.)--(-9.,5.);
\draw (-9.,5.)--(-6.,5.);
\draw (1.,6.)--(4.,6.);
\draw (4.,6.)--(4.,5.);
\draw (4.,5.)--(1.,5.);
\draw (1.,5.)--(1.,6.);
\draw (-9.,5.)--(-9.,4.);
\draw (-9.,4.)--(-6.,4.);
\draw (-6.,4.)--(-6.,5.);
\draw (-6.,4.)--(-4.,4.);
\draw (-4.,4.)--(-4.,5.);
\draw (-4.,4.)--(-2.,4.);
\draw (-2.,4.)--(-2.,5.);
\draw (-2.,4.)--(0.,4.);
\draw (0.,4.)--(0.,5.);
\draw (1.,5.)--(1.,4.);
\draw (1.,4.)--(3.,4.);
\draw (3.,4.)--(3.,5.);
\draw (3.,4.)--(5.,4.);
\draw (4.,5.)--(6.442360843347224,4.950377061577588);
\draw (6.442360843347224,4.950377061577588)--(6.442360843347224,3.996871532432861);
\draw (5.,4.)--(4.999587365006443,4.979690772394813);
\draw (5.,4.)--(6.442360843347224,3.996871532432861);
\draw (4.,6.)--(6.442360843347224,5.9745126299182205);
\draw (6.442360843347224,5.9745126299182205)--(6.442360843347224,4.950377061577588);
\draw (-9.,3.)--(-9.025617740556104,2.390038140726006);
\draw (6.407045823749271,2.972735964092228)--(-9.,3.);
\draw (6.407045823749271,2.972735964092228)--(6.407045823749271,2.4076956505249827);
\draw (-9.025617740556104,2.390038140726006)--(6.407045823749271,2.4076956505249827);
\node[anchor=north west,inner sep=0pt] at (-7.683646995833897,5.656677453536645) {$n$};
\node[anchor=north west,inner sep=0pt] at (-4.946732976992554,5.621362433938692) {$n-1$};
\node[anchor=north west,inner sep=0pt] at (-1.9802713307645194,5.621362433938692) {$n-2$};
\node[anchor=north west,inner sep=0pt] at (-5.494115780760823,4.614884375397035) {$n-1$};
\node[anchor=north west,inner sep=0pt] at (1.568888138829737,5.656677453536645) {$n-q+1$};
\node[anchor=north west,inner sep=0pt] at (-7.71896201543185,4.614884375397035) {$n$};
\node[anchor=north west,inner sep=0pt] at (-1.6271211347849914,4.632541885196012) {$n-3$};
\node[anchor=north west,inner sep=0pt] at (-3.551789702873419,4.614884375397035) {$n-2$};
\node[anchor=north west,inner sep=0pt] at (1.087450393855327,4.6501993949949885) {$n-q'+2$};
\node[anchor=north west,inner sep=0pt] at (3.099979040335566,4.685514414592942) {$n-q'+1$};
\node[anchor=north west,inner sep=0pt] at (4.570664804655725,5.656677453536645) {$n-q$};
\node[anchor=north west,inner sep=0pt] at (5.206335157418875,4.703171924391918) {$n-q'$};
\node[anchor=north west,inner sep=0pt] at (-7.648331976235944,6.574867963083419) {$k$};
\node[anchor=north west,inner sep=0pt] at (-4.717185349605861,6.557210453284442) {$k$};
\node[anchor=north west,inner sep=0pt] at (-1.7330661935788498,6.557210453284442) {$k$};
\node[anchor=north west,inner sep=0pt] at (2.2928460405877695,6.504237923887513) {$k$};
\node[anchor=north west,inner sep=0pt] at (-7.71896201543185,3.80263892464412) {$k$};
\node[anchor=north west,inner sep=0pt] at (-5.511773290559799,3.7849814148451437) {$k-1$};
\node[anchor=north west,inner sep=0pt] at (-3.5341321930744427,3.80263892464412) {$k-1$};
\node[anchor=north west,inner sep=0pt] at (1.4629430800358787,3.837953944242073) {$k-1$};
\node[anchor=north west,inner sep=0pt] at (3.440584177521236,3.80263892464412) {$k-1$};
\node[anchor=north west,inner sep=0pt] at (5.171020137820923,6.362977845495702) {$r$};
\node[anchor=north west,inner sep=0pt] at (5.541827843599427,3.9085839834379787) {$r'$};
\node[anchor=north west,inner sep=0pt] at (0.3152049431024127,5.444787335948928) {$...$};
\node[anchor=north west,inner sep=0pt] at (0.3328624529013891,4.473624297005224) {$...$};
\node[anchor=north west,inner sep=0pt] at (-1.5741486053880622,3.8202964344430965) {$k-1$};
\node[anchor=north west,inner sep=0pt] at (6.15,2.8314758857004168) {$1$};
\node[anchor=north west,inner sep=0pt] at (5.85,2.8314758857004168) {$2$};
\node[anchor=north west,inner sep=0pt] at (3.422926667722259,2.81381837590144) {$l_1$};
\node[anchor=north west,inner sep=0pt] at (-1.168025880011605,2.81381837590144) {$l_2$};
\node[anchor=north west,inner sep=0pt] at (-8.9,2.81381837590144) {$d$};
\node[anchor=north west,inner sep=0pt] at (-8.5,2.81381837590144) {$d-1$};
\node[anchor=north west,inner sep=0pt] at (4.5,2.6549007877106523) {$...$};
\node[anchor=north west,inner sep=0pt] at (0.3681774724993419,2.672558297509629) {$...$};
\node[anchor=north west,inner sep=0pt] at (-4.664212820208932,2.672558297509629) {$...$};
\end{tikzpicture}
\end{center}
\caption{Comparison of $u_{k,d,n}$ with $x_nu_{k-1,d-1,n}$ }
\label{comparison}
\end{figure}

Now since $u_{k-1,d-1,n}\in B^k(u_{k,d,n}):\mm$,  and since by Lemma \ref{hot} the monomials of degree $d-1$ in $B^k(u_{k,d,n}):\mm$ generate $(k-1)$-Borel ideal, we see that $B^{k-1}(u_{k-1,d-1,n})\subset B^k(u_{k,d,n}):\mm$.

Conversely, note that with respect to $\preceq_{k-1}$,  the monomial $u_{k-1,d-1,n}$ is the unique largest $(k-1)$-bounded  monomial of degree $d-1$. Therefore,  $B^k(u_{k,d,n}):\mm\subseteq B^{k-1}(u_{k-1,d-1,n})+B^k(u_{k,d,n})$, because Lemma~\ref{hot} implies that the monomials $B^k(u_{k,d,n}):\mm$ which do not belong to $B^k(u_{k,d,n})$ are $(k-1)$-bounded of degree $d-1$.
\end{proof}

As an immediate consequence of Theorem~\ref{borelsocle} we obtain

\begin{Corollary}
\label{newmafi}
Given integers $k,d,n,$ with  $1\leq d\leq n$ and $k\geq 1$.
 Then
\[
\sat(I_{d,n}^k) =\max\{l\:\; (kd-l)/(k-l) \leq n, \ l\leq k\}.
\]
\end{Corollary}

\begin{Example}
{\em Let $I=(x_1x_2,x_1x_3,x_2x_3)$. Then $\sat(I)=0$ and for $k\geq 2$ we have

\[
\sat(I^k) = \left\{
\begin{array}{ll}
k/2,  &   \text{if $k$ is even,}\\
(k-1)/2, & \text{if $k$ is odd}.
\end{array}
\right.
\]
In fact, $\sat(I)=0$ because $\depth(S/I)>0$. Now let $k\geq 2$. If $k=2a$,
then $(kd-l)/ ( k-l)=(4a-l)/(2a -l)\leq 3$ if and only if  $l\leq a$. Therefore, $a=\max\{l\:\; (kd-l)/(k-l) \leq n, \ l\leq k\}$ and $\sat(I^k)=k/2$. Finally, if $k=2a+1$,  then
$(kd-l)/ ( k-l)=(4a+2-l)/(2a+1-l)\leq 3$ if and only if $l\leq a$. Therefore, $a=\max\{l\:\; (kd-l)/(k-l) \leq n, \ l\leq k\}$ and $\sat(I^k)=(k-1)/2$.
}
\end{Example}

Now let $I$ be any polymatroidal ideal. In \cite[Proposition 5]{HV} it is shown that $I$ is of intersection type, which means that $I$ is the intersection of powers of monomial prime ideals. In other words, there exists monomial prime ideals $P_1,\ldots, P_r$ and positive  integers $a_i$, and $d\geq 0$ such that $I=P_1^{a_1}\sect\ldots\sect P_r^{a_r}\sect \mm^d$. Notice that $I^{\sat}=P_1^{a_1}\sect\ldots\sect P_r^{a_r}$, so that $I=I^{\sat}\sect \mm^d$.

Proposition~\ref{chinesecopy} implies that if  $a$ the least degree of a generator of $P_1^{a_1}\sect\ldots\sect P_r^{a_r}$,  and  $d\geq a$, then $\sat(I)=d-a$.

\begin{Example}
\label{sleepyamir}
{\em Let $I=(x_1,x_2)\sect (x_2,x_3)\sect (x_2,x_4)\sect (x_3,x_4)\sect (x_1,x_2,x_3, x_4)^5$. An element of least degree in $(x_1,x_2)\sect (x_2,x_3)\sect (x_2,x_4)\sect (x_3,x_4)$ is $x_2x_3$. Therefore, $\sat(I)=5-2=3$.}
\end{Example}

By applying monomial localization one obtains

\begin{Corollary}
\label{vacationbegins}
Let $I$ be a monomial ideal with the property that $\Ass(I)=\Ass^\infty(I)$  and that all powers of $I$ are of intersection type.  Then $\sat(I^k)=\sat(I)k$ for all $k$.
\end{Corollary}

Since power of polymatroidal ideals are again polymatroidal we have

\begin{Corollary}
\label{poly}
Let $I$ be a polymatroidal ideal with $\Ass(I)=\Ass^\infty(I)$.  Then  $\sat(I^k)=\sat(I)k$ for all $k$.
\end{Corollary}

\begin{Example}{\em
Let $I$ be a transversal polymatroid.  In other words, $I$ is a product of monomial prime ideals. Then $\sat(I^k)=\sat(I)k$. This follows from Corollary~\ref{poly}  because $\Ass(I)=\Ass^\infty(I)$, (see\cite[Corollary 4.6]{HRV}).}
\end{Example}

In order to compute $\sat(I)$ of a polymatroidal ideal we have to determine its presentation as an intersection of   powers of monomial prime ideals, as described in \cite{HV}: let $\P$ be a discrete polymatroid on the ground set $[n]$  of rank $d$ with rank function $\rho$, see \cite{HH}.  The {\em complementary rank function}  $\tau\: 2^{[n]}\to \ZZ_+$ is given by  $\tau(F)=d-\rho([n]\setminus F)$ for all $F\in 2^{[n]}$.

A subset $F\subset [n]$ is called {\em  $\tau$-closed}, if $\tau (G)<\tau(F)$ for any proper subset $G$ of $F$, and $F$ is  called {\em $\tau$-separable} if there exist non-empty subsets $G$ and $H$ of $F$ with $G\sect H=\emptyset$ and $G\union H=F$ such that $\tau(G)+\tau(H)=\tau(F)$. If $F$ is not $\tau$-separable, then it is called $\tau$-inseparable.

With this information the  intersection presentation of polymatroidal ideal is given as follows:

\begin{Theorem}[Theorem 12,  \cite{HV}]
\label{redundant}
Let $I$ be a polymatroidal ideal associated with the discrete polymatroid $\mathcal P$ with complementary rank function $\tau$. Then
\[
I=\Sect_{F}P_F^{\tau(F)},
\]
where the intersection is taken over all $F\subset [n]$ which are $\tau$-closed and $\tau$-inseparable.
\end{Theorem}

\begin{Example}{\em Consider the Veronese type ideal $I_{3,2,\ldots,2}\subset K[x_1,\ldots,x_n]$. Its rank function is given by $\rho(\emptyset)=0$, $\rho(A)=2$ if $|A|=1$ and  $\rho(A)=3$ if $|A|\geq 2$. Therefore, $\tau(A)=0$ if $|A|\leq n-2$,  $\tau(A)=1$ if $|A|=n-1$ and  $\tau(A)=3$ if $|A|=n$.  It follows that $A$ is $\tau$-closed and $\tau$-inseparable, if and only if $|A|=n-1$ or $|A|=n$. Thus Theorem~\ref{redundant} implies that
\[
I_{3,2,\ldots,2}=\Sect_{A, |A|=n-1}P_A\sect \mm^3.
\]
Since $x_1x_2\in \Sect_{A,  |A|=n-1}P_A$, Corollary~\ref{chinesecopy} implies that $\sat(I_{3,2,\ldots,2})=1$.
}
\end{Example}


\begin{thebibliography}{99}
\bibitem{AB} J. Abbott and A. M. Bigatti, {\it a C++ library for doing Computations in Commutative Algebra}, Available at
    http://cocoa.dima.unige.it/cocoalib.
 \bibitem{AHZ} R. Abdolmaleki, J. Herzog and G. Zhu, {\it The saturation number of $c$-bounded stable monomial ideals and their powers}, Kyoto J. Math., {\bf 62}(2022), 471-484.    
\bibitem{B} M. Brodmann, {\it The asymptotic nature of the analytic spread}, Math. Proc. Cambridge Philos. Soc., {\bf 86}(1979), 35-39.
\bibitem{BH} W. Bruns and J. Herzog, {\it Cohen-Macaulay rings}, Cambridge University Press, Cambridge, UK, (1998).
\bibitem{CH} A. Conca and J. Herzog, {\it Castelnuovo-Mumford regularity of products of ideals}, Collect. Math., {\bf 54}(2003), 137-152.
\bibitem{CHT} D. Cutkosky, J. Herzog and N.V. Trung, {\it Asymptotic behaviour of Castelnuovo-Mumford regularity},
Compos. Math., {\bf 80}(1999), 273-297.
\bibitem{E} D. Eisenbud, {\it Commutative algebra}, with a view toward algebraic geometry, Graduate Texts in Mathematics, 150, Springer-Verlag, New York, (1995).
\bibitem{GS}D. R. Grayson and M. E. Stillman, {\it Macaulay 2, a software system for research in algebraic geometry}, Available at
    {http://www.math.uiuc.edu/Macaulay2/}.
\bibitem{GSp} S. Ginsburg and E. H. Spanier, {\it Semigroups, Presburger formulas, and languages}, Pacific J. Math., {\bf 16}(1966), 285-296.
\bibitem{HH} J. Herzog and T. Hibi, {\it Discrete polymatroids}, J. Algebr. Comb., {\bf 16}(2002), 239-268.
\bibitem{HH2} J. Herzog and T. Hibi, {\it Monomial ideals}, Grad. Texts Math., vol.260, Springer-Verlag London, Ltd., London, (2011).
\bibitem{HHT} J. Herzog and T. Hibi and N. V. Trung, {\it Symbolic powers of monomial ideals
and vertex cover algebras}., Adv. Math. {\bf 210}(2007), 304-322.
\bibitem{HLR} J. Herzog, B. Lajmiri and F. Rahmati, {\it On the associated prime ideal and the depth of powers of squarefree principal Borel ideals}., to appear in Inter. Electron. J. Algebr., {\bf 26}(2019).
\bibitem{HPV} J. Herzog, D. Popescu and M. Vladoiu, {\it On the Ext-modules of ideals of Borel type}, Contem.
Math., {\bf 331}(2003), 171-186.
\bibitem{HRV} J. Herzog, A. Rauf and M. Vladoiu, {\it The stable set of associated prime ideals of a polymatroidal ideal}., J. Algebr. Comb., {\bf 37}(2013), 289-312.
\bibitem{HT} J. Herzog and Y. Takayama, {\it Resolutions by mapping cones}, Homology Homotopy Appl., {\bf 4}(2002), 277-294.
\bibitem{HV} J. Herzog and M. Vladoiu, {\it Monomial ideals with primary components given by powers of monomial prime ideals},
Electron. J. Comb., {\bf 21}(2014), P1.69.
\bibitem{Ko} V. Kodiyalam, {\it Asymptotic behaviour of the Castelnuovo-Mumford regularity}, Proc. AMS., {\bf 128}(1999), 407-411.
\end{thebibliography}
\end{document}